\documentclass[11pt,a4paper]{article}
\include{latin}

\usepackage{nomencl}

\usepackage{kerkis}

\usepackage[square,sort,comma,numbers]{natbib}
\usepackage{amsmath}
\usepackage{amsfonts,amscd}
\usepackage[all]{xy}
\usepackage{amssymb}
\usepackage{makeidx}
\usepackage{graphicx}
\usepackage{tikz}
\usetikzlibrary{matrix}
\usepackage{framed}
\usepackage{wasysym}

\usepackage[...]{youngtab}

\usepackage{geometry}
\usepackage{amsthm}

\newtheorem{theorem}{Theorem}[section]
\newtheorem{lemma}[theorem]{Lemma}
\newtheorem{proposition}[theorem]{Proposition}
\newtheorem{corollary}[theorem]{Corollary}
\begin{document}
\title{The normal closure of a power of a half-twist has infinite index in the mapping class group of a punctured sphere}
\author{Charalampos Stylianakis}
\date{ }
\maketitle

\begin{abstract}
\textit{In this paper we show that the normal closure of the m$^{th}$ power of a half-twist has infinite index in the mapping class group of a punctured sphere. Furthermore, in some cases we prove that the quotient of the mapping class group of the punctured sphere by the normal closure of a power of a half-twist contains a free abelian subgroup. As a corollary we prove that the quotient of the hyperelliptic mapping class group of a surface of genus at least two by the normal closure of the m$^{th}$ power of a Dehn twist has infinite order, and for some integers $m$ the quotient contains a free nonabelian subgroup. As a second corollary we recover a result of Coxeter: the normal closure of the m$^{th}$ power of a half-twist in the braid group of at least five strands has infinite index if n is at least four. Our method is to reformulate the Jones representation of the mapping class group of a punctured sphere, using the action of Hecke algebras on W-graphs, as introduced by Kazhdan-Lusztig.}
\end{abstract}

\section{Introduction}

Let $\Sigma_{g,n}$ be a surface of genus $g$ with $n$ marked points. The mapping class group of $\Sigma_{g,n}$, denoted by $\mathrm{Mod}(\Sigma_{g,n})$, consists of those isotopy classes of homeomorphisms that preserve both the orientation of $\Sigma_{g,n}$, and the set of marked points. When $n=0$, we will simply write $\Sigma_g$. If $g=0$ the mapping class group $\mathrm{Mod}(\Sigma_{0,n})$ is generated by half-twists (homeomorphisms that interchange two marked points of $\Sigma_{0,n}$). The main result of this paper is the following.

\paragraph{Theorem A} \textit{The normal closure of the $m^{th}$ power of a half-twist has infinite index in $\mathrm{Mod}(\Sigma_{0,n})$ if $n\geq6$ is even and $m\geq5$.}\\

The proof of Theorem A relies on the Jones representation of the mapping class group $\mathrm{Mod}(\Sigma_{0,n})$, where $n\geq 6$ is even \cite{JO}. The following theorem was suggested by Funar.

\paragraph{Theorem B} \textit{The quotient $\mathrm{Mod}(\Sigma_{0,n})$ by the closure of the $m^{th}$ power of a half-twist contains a free nonabelian subgroup, if $n\geq6$ is even, and $m\notin \{ 2,4,6,10\}$ if $m$ is even, and $m \notin \{ 1,3,5 \}$ if $m$ is odd.}\\

The proof of Theorem B we give in this paper uses the Jones representation of the mapping class group of the sphere with marked points. As pointed out by Funar, there is an alternate proof of Theorem B in the case where $m$ even. Funar-Kohno used quantum representations to prove that the quotient of the mapping class group of the sphere with 4 boundary components by the normal closure of the $k^{th}$ power of Dehn twists contains free nonabelian subgroup \cite[Proposition 3.2]{FK}. Therefore, the quotient of the mapping class group of the sphere with $n$ boundary components by the normal closure of the $k^{th}$ power of Dehn twists contains free nonabelian subgroup. The surjection of the latter quotient on the quotient $\mathrm{Mod}(\Sigma_{0,n})$ by the closure of the $2k^{th}$ power of a half-twist has finite kernel generated by the Dehn twists about the boundary components; the image contains a free nonabelian subgroup.

\paragraph{Results for the hyperelliptic mapping class group.} Let $c$ be a nonseparating simple closed curve in $\Sigma_g$; denote by $T_c$ a Dehn twist about $c$. In Birman's problem list, she asked whether the normal closure of $T^2_c$ in $\mathrm{Mod}(\Sigma_g)$ has infinite index if $g\geq 3$ \cite[Question 28]{JB}. It is well known that the normal closure of $T^2_c$ has finite index in $\mathrm{Mod}(\Sigma_g)$ when $g=1$ or 2.

Humphries answered Birman's question by proving that, in fact, the normal closure of $T^2_c$ has finite index in $\mathrm{Mod}(\Sigma_g)$ \cite[Theorem 1]{HU}. Let $\mathrm{SMod}(\Sigma_g)$ denote the hyperelliptic mapping class group, that is, those elements of $\mathrm{Mod}(\Sigma_g)$ that commute with a fixed hyperelliptic involution (an element of order 2 of $\mathrm{Mod}(\Sigma_g)$ that acts as $-\mathrm{Id}$ on $\mathrm{H}_1(\Sigma_g)$). In the same paper Humphries used the fact that $\mathrm{Mod}(\Sigma_2)=\mathrm{SMod}(\Sigma_2)$ to show that if $m\geq 4$, then the normal closure of $T^m_c$ has infinite index in $\mathrm{Mod}(\Sigma_2)$ \cite[Theorem 4]{HU}. As a corollary of the Theorem A, we extend Humphries' result for $\mathrm{SMod}(\Sigma_g)$ to the case $g > 2$.

\begin{corollary}
\label{main theorem}
Let $T_c \in \mathrm{SMod}(\Sigma_g)$ be a Dehn twist about a nonseparating simple closed curve $c$. If $m \geq 5$, and $g\geq 2$ then the normal closure of $T^m_c$ has infinite index in $\mathrm{SMod}(\Sigma_g)$.
\end{corollary}

We note that the curve $c$ in the theorem above is symmetric, that is, it remains fixed under the action of a fixed hyperelliptic involution.

\begin{proof}
By Birman-Hilden's theorem we have that the map  $\mathrm{SMod}(\Sigma_g) \rightarrow \mathrm{Mod}(\Sigma_{0,2g+2})$ is a surjective homomorphism \cite[Theorem 1]{BH}. The result follows from the fact that the normal closure of $T^n_c$ in $\mathrm{SMod}(\Sigma_g)$ surjects onto the normal closure of a half-twist in $\mathrm{Mod}(\Sigma_{0,n})$. 
\end{proof}

%Humphries' proof of the latter result above relies on the Jones representation of $\mathrm{SMod}(\Sigma_g)$ \cite[Section 10]{JO}. In particular, Humphries makes use of Jones' explicit calculation of this representation when $g=2$. Unfortunately, the explicit computation of the matrices when $g>2$ is very difficult. 

Our methods in the proof of Theorem A are inspired by Humphries' original proof, which relies on Jones' explicit calculation of his representation of $\mathrm{Mod}(\Sigma_{0,2g+2})$ when $g = 2$.  Unfortunately, the explicit computation of the matrices when g > 2 is difficult. We avoid this difficulty by finding a certain block form of the representation in the general case.

We note that it is not obvious how to extend the Jones representation from $\mathrm{SMod}(\Sigma_g)$ to the whole mapping class group $\mathrm{Mod}(\Sigma_g)$ (see for example \cite[Section 2.3]{KA}), and hence Humphries' techniques cannot be used for $\mathrm{Mod}(\Sigma_g)$ when $g>2$. However, Funar successfully used TQFT representations of the mapping class group to prove that the normal closure of $T^m_c$ has infinite index in $\mathrm{Mod}(\Sigma_g)$, if $g \geq 2$ and if $m \neq 2,3,4,6,8,12$ \cite[Corollary 1.2]{FU}.

Funar-Kohno improved Funar's result above by proving that if $g\geq 3$, and $m \notin \{ 3, 4, 8, 12, 16, 24 \}$ or $g= 2$, $m$ even and $m \notin \{ 4, 8, 12, 16, 24, 40 \}$ the quotient of $\mathrm{Mod}(\Sigma_g)$ by the normal closure of the $m^{th}$ power of a Dehn twist contains a free nonabelian subgroup \cite{FK}. We can use Theorem B to prove a similar theorem for the hyperelliptic mapping class group.

\begin{corollary}
We assume that $g\geq 2$, $m\geq 4$, and $m\notin \{ 1,2,3,4,6,10\}$ is even. Then the quotient of $\mathrm{SMod}(\Sigma_g)$ by the normal closure of the $m^{th}$ power of a Dehn twist contains a free nonabelian subgroup.
\end{corollary}

The proof of Corollary 1.2 is the same as the proof of Corollary 1.1.

\paragraph{Remarks on Theorems A and B.} Theorem B gives a stronger result than Theorem A in the case where the power $m$ is even. The methods we use to prove them are different. The proof of Theorem A uses an explicit matrix calculation that we develop in this paper. In fact we give a new approach to the Jones representation by using a different definition of the Hecke algebras and we use the notion of W-graphs introduced by Kazhdan-Lusztig. In the proof of Theorem B we show how to modify the Burau representation so that its image is contained in the Jones representation.

Quantum representations deal with the quotient of the mapping class group of the sphere with boundary components by the normal closure of powers of Dehn twists. The representations we construct in this paper deal with representations of the quotient of the mapping class group $\mathrm{Mod}(\Sigma_{0,n})$ of the sphere with even number of marked points by the normal closure of powers of half-twists. One could find relations between the two representations, and this may give new representations of the quotient of $\mathrm{Mod}(\Sigma_{0,n})$ by the normal closure of powers of half-twists. These representation might give more information about the latter quotient.

\paragraph{Results for braid groups.} We denote the disc with $n$ marked points in its interior by $D_n$. The mapping class group $\mathrm{Mod}(D_n)$ of $D_n$ consists of isotopy classes of homeomorphisms that preserve both the orientation of $D_n$, the set of marked points, and fix the boundary pointwise. The braid group $B_n$ is isomorphic to $\mathrm{Mod}(D_n)$. It is a well known result due to Artin that $B_n$ is generated by half-twists, that is, homeomorphisms that interchange two marked points \cite[Sections 1.2, 1.3]{BB}.

%Furthermore, we can think of the pure braid group $PB_n$ as the subgroup of $B_n$ generated by Dehn twists about curves surrounding two points in $D_n$. Humphries gave conditions of when a group generated by possibly different powers of Dehn twists has finite index in $PB_n$ \cite[Theorems 1,2,3]{HUM}.

A similar result as in Theorem A holds for the braid group $B_n$. Coxeter used hyperbolic geometry to prove that the normal closure of the $m^{th}$ power of a half-twist has finite index in $B_n$ if and only if $(n-2)(m-2)<4$ \cite[Section 10]{C2}. As a second corollary of our Theorem A, we recover Coxeter's theorem when $n\geq4$ and $m\geq 4$.
\begin{corollary}
The normal closure of the $m^{th}$ power of a half-twist has infinite index in the braid group $B_n$, if $m \geq 5$, and $n \geq 4$.
\label{coxthem}
\end{corollary}

The proof of Corollary \ref{coxthem} is independent from Coxeter's proof.
%In this paper we study the Hecke algebra representations in the viewpoint of Kashdan-Lusztig \cite{KL} to extend Humphries' result for $\mathrm{SMod}(\Sigma_g)$.

\paragraph{Construction of the Jones representation.} Let $H(q,2g+2)$ be a Hecke algebra with complex parameter $q$. There is a representation $B_{2g+2} \rightarrow H(q,2g+2)$ from the braid group into the group of units of $H(q,2g+2)$. We can think of $H(q,2g+2)$ as a quotient of the group algebra of $B_{2g+2}$ over $\mathbb{Z}[q^{\pm 1}]$. Thus, any representation of $H(q,2g+2)$ will give a representation for $B_{2g+2}$. We can think of $\mathrm{Mod}(\Sigma_{0,2g+2})$ as a quotient group of the braid group $B_{2g+2}$. Jones observed that in some cases we can modify the representations of $H(q,2g+2)$ so that we can define representations for $\mathrm{Mod}(\Sigma_{0,2g+2})$ \cite[Section 10]{JO}.

Assume that $q$ is not a root of unity. The set of irreducible representations of the Hecke algebra $H(q,2g+2)$ is in bijective correspondence with the set of Young diagrams of size $2g+2$. When the Young diagram has the shape of a rectangle, we will show that under a modification, the corresponding irreducible representation of $H(q,2g+2)$ gives a representation of $\mathrm{Mod}(\Sigma_{0,2g+2})$. We will also explain a method for explicitly computing matrices of the irreducible representations of $H(q,2g+2)$ in this case by using the notion of W-graphs (see Section 3.2). If $g=2$ we explicitly calculate the matrices of the representation of $\mathrm{Mod}(\Sigma_{0,2g+2})$. Our calculations are equivalent to those of Jones, but we make different choices of parameters and hence the resulting matrices are slightly different. When $g \geq 3$, the calculations are much more complicated and we will not compute the full matrices explicitly. However, we will show that the matrices have a particular block form for $g\geq 3$ that is sufficient for the calculations we require (see Theorem \ref{main3}).

\paragraph{Outline of the paper.} In Section 2 we give basic background for braid groups, hyperelliptic mapping class groups, and Hecke algebras. In Section 3 we define representations of Hecke algebras, and consequently of braid groups. More particularly, in Section 3.1 we give relations between irreducible representations of Hecke algebras, and Young diagrams. Sections 3.2 and 3.3 are devoted to the construction of $\mathbb{Z}[q^{\pm 1}]$-modules, and an action of Hecke algebras on these modules by using the notion of W-graphs; as a result we end up with well-defined irreducible representations of Hecke algebras. In Section 3.3 we explain how Young diagrams are related to these $\mathbb{Z}[q^{\pm 1}]$-modules. In Section 3.4 we examine certain properties of representations of $B_n$ we obtain via the homomorphism $B_n \rightarrow H(q,n)$. In Section 4 we define the representation of $\mathrm{SMod}(\Sigma_{0,2g+2})$. In Section 5 we prove Theorem A, and Corollary \ref{coxthem}. In Section 6 we prove Theorem B.

\paragraph{Acknowledgements.} The author is grateful to his PhD supervisor Tara Brendle for her guidance throughout this work. In addition the author would like to thank Gwyn Bellamy, Joan Birman and Louis Funar for comments on an earlier draft and Louis Funar for his suggestion of Theorem B. The author would also like to thank Duncan McCoy for helpful conversations.

\section{Preliminaries}

In this section we give basic facts about braid groups, Hecke algebras and the hyperelliptic mapping class group.

\paragraph{The braid group and mapping class group of the punctured sphere.} Consider a disc $D_n$ with $n$ distinct marked points in its interior as indicated in Figure 1. The braid group $B_n$ is defined to be $\mathrm{Mod}(D_n)$. We denote by $p_i$ the marked points of $D_n$, enumerating from left to right. For $i \leq n-1$ we denote by $\sigma_i$ the mapping classes of $D_n$ obtained by interchanging the points $p_i$ and $p_{i+1}$ by rotating clockwise as depicted in Figure 1.
\begin{figure}[h]
\begin{center}
\includegraphics[scale=.2]{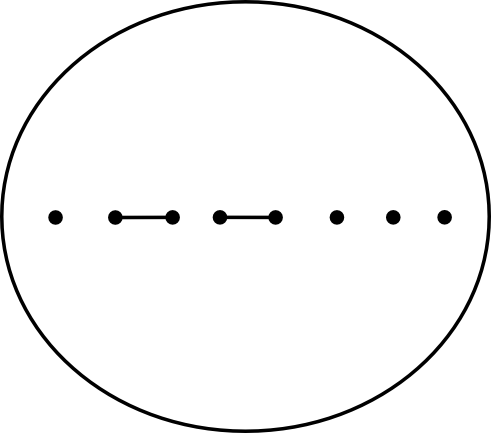}
\includegraphics[scale=.2]{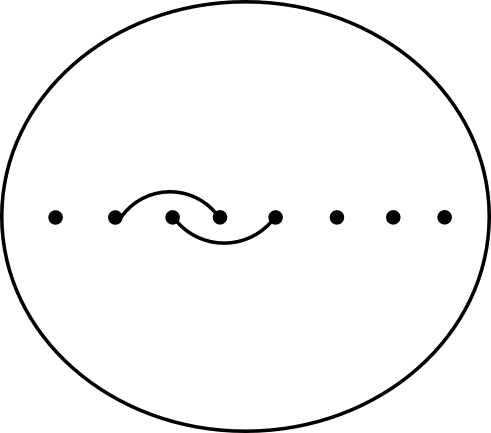}
\end{center}
\label{sigmagen}
\caption{The action of the half-twists $\sigma_3$ on two arcs connecting points $p_2,p_3$ and $p_4,p_5$ respectively.}
\begin{picture}(22,12)

\end{picture}
\end{figure}
The elements $\sigma_i$ are called \emph{half-twists}, and they generate $B_n$. More precisely, we have the following presentation for $B_n$:
\[ \langle \sigma_i,...,\sigma_{n-1} \mid \sigma_i \sigma_{i+1} \sigma_i = \sigma_{i+1} \sigma_i \sigma_{i+1}, \, [\sigma_i, \sigma_j]=1 \, \text{if} \, |i-j|>1 \rangle . \]
For more details about definitions and presentations of $B_n$ see \cite[Chapter 1]{BB}.

Consider the natural epimorphism $B_n \rightarrow S_n$, where $S_n$ is the symmetric group. More precisely, we have $B_n / \langle \sigma^2_i \rangle \cong S_n$. We denote by $s_i$ the transpositions $(i,i+1)$ in $S_n$. We have the following presentation for $S_n$:
\[ \langle s_1,...,s_{n-1} \mid s^2_i=1, s_i s_{i+1} s_i = s_{i+1} s_i s_{i+1}, \, [s_i, s_j]=1 \, \text{if} \, |i-j|>1 \rangle . \]
We denote by $S$ the set $\{s_i\mid 1\leq i\leq n-1\}$; the pair $(S_n,S)$ forms a \emph{Coxeter system} for the symmetric group (see \cite{BB1} for general definition).

Consider the sphere with $n$ punctures $\Sigma_{0,n}$ induced from $D_n$ by gluing a disc to the boundary of $D_n$. By extending every homeomorphism of $D_n$ as the identity in $\Sigma_{0,n}$ we obtain a surjective homomorphism
\[ B_n \rightarrow \mathrm{Mod}(\Sigma_{0,n}). \]
Every half-twist in $B_n$ is mapped into a half-twist in $\mathrm{Mod}(\Sigma_{0,n})$. We will denote the generators of $\mathrm{Mod}(\Sigma_{0,n})$ by $H_i$. Thusly, the homomorphism above is defined by $\sigma_i \mapsto H_i$.

The group $\mathrm{Mod}(\Sigma_{0,n})$ has a presentation with generators $H_1,H_2,...,H_{n-1}$, and relations
$$H_i H_{i+1} H_i = H_{i+1} H_i H_{i+1}$$ $$H_k H_j = H_j H_k, \: \mathrm{if} \: \vert k-j \vert>1$$ $$(H_1 H_2...H_{n-1})^{n}=1$$ and $$H_1 H_2...H^2_{n-1}...H_2 H_1 = 1.$$
where $1 \leq i \leq n-2$ and $1\leq k,j \leq n-1$ \cite[Section 5.1.3]{BFM}

\paragraph{Hyperelliptic mapping class group.} The mapping class group $\mathrm{Mod}(\Sigma_g)$ is generated by Dehn twists about the curves $ c_0,c_1,c_2,...,c_{2g}$ depicted in Figure \ref{Hypmap} \cite{HUS}.

\begin{figure}[h]
\begin{center}
\includegraphics[scale=.3]{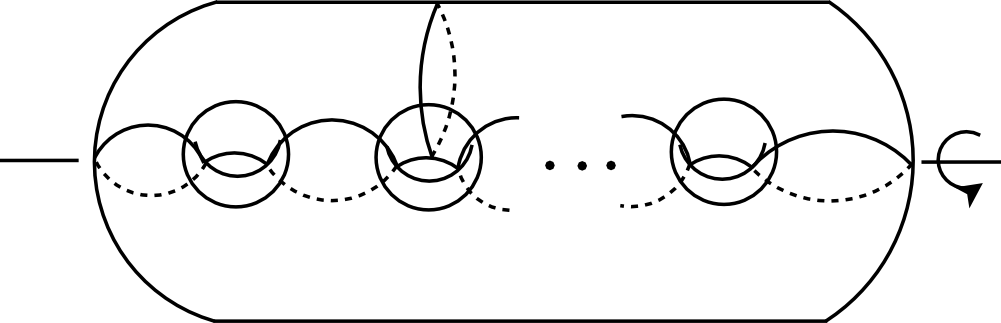}
\end{center}
\caption{Generators for $\mathrm{Mod}(\Sigma_g)$.}
\label{Hypmap}
\begin{picture}(22,12)
\put(125,71){$c_1$}
\put(145,71){$c_2$}
\put(167,71){$c_3$}
\put(190,71){$c_4$}
\put(181,109){$c_0$}
\put(260,71){$c_{2g}$}
\put(280,71){$c_{2g+1}$}
\put(330,78){$\iota$}
\end{picture}
\end{figure}

We think of $\Sigma_g$ as a branched cover of the sphere $\mathbb{S}^2$, over $2g+2$ points as follows:
let $\iota$ be the hyperelliptic involution of the surface $\Sigma_g$ as indicated in Figure \ref{Hypmap}. Obviously $\iota \in \mathrm{Mod}(\Sigma_g)$. Then the quotient $\Sigma_g / \iota$ is an orbifold sphere with $2g+2$ cone points of order 2. In this paper we will not need this geometric information about $\Sigma_g / \iota$, and we will consider it as a topological space $\Sigma_{0,2g+2}$, that is, a sphere with $2g+2$ marked points.

We denote by $\mathrm{SMod}(\Sigma_g)$ the subgroup of $\mathrm{Mod}(\Sigma_g)$ consisting of mapping classes that commute with $\iota$. Birman-Hilden proved the following exact sequence \cite[Theorem 1]{BH}.
\[ 1 \rightarrow \langle \iota \rangle \rightarrow \mathrm{SMod}(\Sigma_g) \rightarrow \mathrm{Mod}(\Sigma_{0,2g+2}) \rightarrow 1 \]
The group $\mathrm{SMod}(\Sigma_g)$ is called \emph{hyperelliptic mapping class group}, and it is generated by $T_{c_1},T_{c_2},...,T_{c_{2g+1}}$. The map $ \mathrm{SMod}(\Sigma_g) \rightarrow \mathrm{Mod}(\Sigma_{0,2g+2})$ is defined by $T_{c_i} \mapsto H_i$.

For $0\leq i \leq 2g+1$ consider the curves $c_i$ indicated in Figure \ref{Hypmap}. Then it is obvious that if $g=2$, then $\mathrm{SMod}(\Sigma_2) = \mathrm{Mod}(\Sigma_2)$. Also, since $T^m_{c_0}$ is not an element of $\mathrm{SMod}(\Sigma_g)$ for any $m \geq 1$, then $\mathrm{SMod}(\Sigma_g)$ has infinite index in $\mathrm{Mod}(\Sigma_g)$.

\paragraph{Hecke Algebras.} Let $\mathbb{Z}[q^{\pm 1}]$ be the ring of Laurent polynomials over $\mathbb{Z}$ where $q$ is an indeterminate. Consider the Coxeter system for the symmetric group $(S_n,S)$ described above, where $S$ consists of the transpositions $s_i = (i,i+1)$. We denote by $H(q,n)$ the algebra over $\mathbb{Z}[q^{\pm 1}]$ generated by the set $\{T_{s_i} \mid s_i \in S \}$ with relations as follows:
\begin{itemize}
	\item[(1)] We have $ T_{s_i} T_{s_{i+1}} T_{s_i} = T_{s_{i+1}} T_{s_i} T_{s_{i+1}}$.
	\item[(2)] If $|i-j|>1$, then $T_{s_i} T_{s_j}=T_{s_j} T_{s_i}$.
	\item[(3)] For all $s_i \in S$, we have $T^2_{s_i} = 1+(q-q^{-1})T_{s_i}$.
\end{itemize}

An algebra $H(q,n)$ of this form is called a \emph{Hecke algebra}. Let $(S_n,S)$ be a Coxeter system. Given $w \in S_n$ we can write $w=s_{j_1} s_{j_2} ... s_{j_p}$, where $s_{j_k} \in S$. If $p$ is minimal, we say that this is a \emph{reduced expression} for $w$; then $l(w)=p$ is called the length of $w$. IF $l(w)=p$ then we define $T_{w} = T_{s_{j_1}} T_{s_{j_2}} ... T_{s_{j_p}}$. The element $T_w$ is independent of the choice of the reduced expression for $w$. It turns out that the multiplication rule is the following, \cite[Lemma 4.4.3]{GP}:
\[
       T_{s_i} T_w = 
  \begin{cases} 
      \hfill  T_{s_i w}, \hfill & \text{ if $l(s_i w) > l(w)$}, \\
      \hfill  T_{s_i w} + (q-q^{-1})T_w, \hfill & \text{ if $l(s_i w) < l(w)$}.\\
  \end{cases}
\]
Furthermore, $H(q,n)$ admits a basis $\{T_w \mid w \in S_n \}$ \cite[Theorem 4.4.6]{GP}.

We can easily check that $T^{-1}_{s_i} = T_{s_1}-(q-q^{-1})$. The map $\psi : B_n \rightarrow H(q,n)$ defined by $\psi( \sigma_i) = T_{s_i}$ is a well defined homomorphism from $B_n$ to the group of units of $H(q,n)$. Similarly if $q=1$ then we have a well defined homomorphism $\phi : S_n \rightarrow H(1,n)$ defined by $\phi(s_i)=T_{s_i}$. We can think of $H(q,n)$ as a quotient of group algebra of $B_n$, and $H(1,n)$ as the group algebra of $S_n$.

%\paragraph{Remark.} In the literature people usually define the Hecke algebra $H(q,n)$ with generators $\tilde{T}_{s_i}$, relations 1, 2 as above, and
%\begin{itemize}
%\item[(3')] $\tilde{T}^2_{s_i}= q+(q-1)\tilde{T}_{s_i}.$
%\end{itemize}
%However, if we change $\tilde{T}_{s_i}$ to $q^{1/2}T_{s_i}$, then the relation 3' becomes 
%$$T^2_{s_i} = 1+(q^{1/2}-q^{-{1/2}})T_{s_i}.$$
%By changing $q^{1/2}$ to $q$ we obtain the relation $(3)$ above.

\section{Braid group representations}

Our aim in this section is to define representations of $H(q,n)$. Since $B_n \rightarrow H(q,n)$ is a well defined homomorphism, any linear representation of $H(q,n)$ will give a rise to a representation of $B_n$. It is known that the set of irreducible representations of $H(q,n)$ is in bijective correspondence with the set of irreducible representations of $S_n$ \cite[Theorem 8.1.7]{GP}. But irreducible representations of $S_n$ are in bijective correspondence with Young diagrams \cite[Theorem 1]{diaconis}. The construction of the representations of $H(q,n)$ we are going to provide arises from the notion of W-graphs introduced by Kazhdan-Lusztig \cite{KL}. This representation is well defined for any $q$, even for roots of unity. For $i < n$ the generators of $S_n$ and $B_n$ will be denoted by $s_i$ and $\sigma_i$ respectively. As mentioned in Section 2, the generators $s_i$ of $S_n$ correspond to the transpositions $(i,i+1)$.

\subsection{Young diagrams}

A \emph{Young diagram} $\lambda = [\mu_1, \mu_2,...,\mu_k]$ is an array of $n$ boxes with $\mu_i$ boxes in the $i^{th}$ row, $\mu_i \geq \mu_{i+1}$, and $\sum \mu_i=n$. We denote by $\Lambda_n$ the set of all Young diagrams with $n$ boxes, and let $\lambda \in \Lambda_n$. A \emph{standard tableau} of $\lambda$ is obtained by filling the boxes of $\lambda$ with integers between 1 and $n$, such that the integers are strictly increasing from left to right and from top to bottom, and every box contains exactly one number. An example is given in Figure \ref{stand}. We denote by $Y_{\lambda}$ the set of all standard tableaux of $\lambda$. For each $Y_{\lambda}$ there is an irreducible representation of $S_n$. The dimension of the representation of the symmetric group $S_n$ associated to $\lambda \in Y_{\lambda}$ is equal to the dimension of the representation of the Hecke algebra $H(q,n)$ \cite[Section 4]{JO}.
\begin{figure}[h]
\begin{center}
\includegraphics[scale=.4]{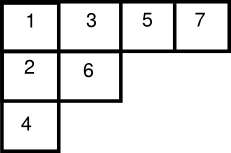}
\end{center}
\caption{Example of a standard tableau.}
\label{stand}
\begin{picture}(22,12)

\end{picture}
\end{figure} 

Let $V_{\lambda}$ be a free $\mathbb{Z}[q^{\pm 1}]$-module with basis $\{ u_t, \mid t \in Y_{\lambda} \}$. The irreducible representation associated to the Young diagram $\lambda$ will be denoted by $\pi_{\lambda}:H(q,n) \rightarrow \mathrm{End}(V_{\lambda})$. We will describe how such an irreducible representation decomposes when it is restricted to $H(q,n-1) \subset H(q,n)$. A \emph{Young's lattice} is a diagram formed by Young diagrams such that each Young diagram is connected by an edge to another one if they differ by one box. Consider for example the Young' lattice indicated in Figure \ref{partition}. 

\begin{figure}[h]
\begin{center}
\includegraphics[scale=.3]{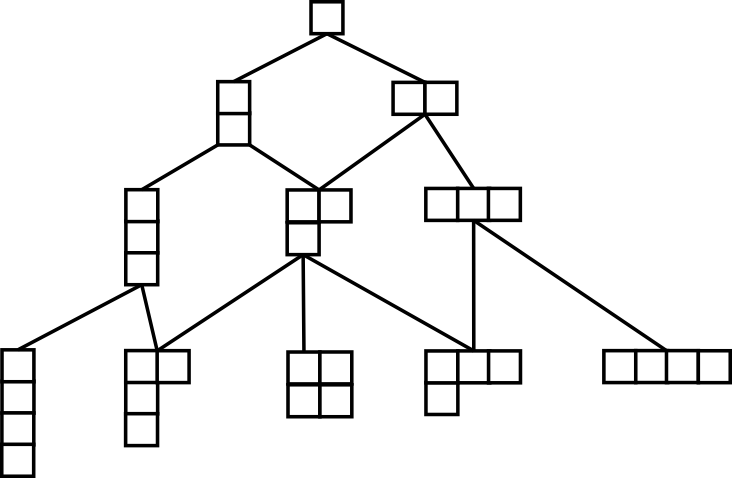}
\end{center}
\caption{Partitions up to 4 boxes.}
\label{partition}
\begin{picture}(22,12)

\end{picture}
\end{figure}
Let $\lambda \in \Lambda_n$ be a diagram which is connected by edges to diagrams $\lambda_1,...,\lambda_m$ such that $\lambda_i \in \Lambda_{n-1}$ for $i =1,2,...,m$. The restriction of the representation $\pi_{\lambda}:H(q,n) \rightarrow \mathrm{End}(V_{\lambda})$ to $H(q,n-1)$ is as follows \cite[Section 4]{JO}:

\[ \bigoplus^{m}_{i=1}\pi_{\lambda_i} :H(q,n-1) \rightarrow  \bigoplus^{m}_{i=1} \mathrm{End}(V_{\lambda_i}).\]

The restriction formula above is called \emph{branching rule}. The dimension of the representation of $H(q,n)$ associated to $\lambda \in \Lambda_n$ is equal to the number of the descending paths from $\Square$ to $\lambda$, which is equal to the cardinality of $Y_{\lambda}$. In order to compute the cardinality of $Y_{\lambda}$ we need the notion of the hook length \cite[Theorem 1]{FRT}. A \emph{hook length} $hook(x)$ of a box $x$ in $\lambda$ is the number of boxes that are in the same row to the right of it plus those boxes in the same column below it, plus one. For example consider the Young diagram of Figure \ref{hook}, where in each box we have assigned its hook length. The cardinality of $Y_{\lambda}$ is equal to

\[ \vert Y_{\lambda} \vert = \frac{n!}{\prod\limits_{x\in \lambda} hook(x)}. \]
For the diagram $\lambda$ of Figure \ref{hook} we have $\vert Y_{\lambda} \vert = 68640$.

\begin{figure}[h]
\begin{center}
\includegraphics[scale=.5]{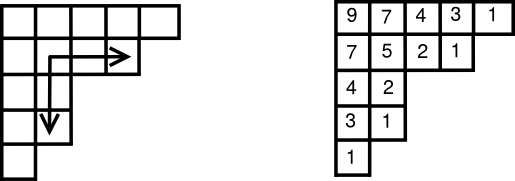}
\end{center}
\caption{The hook length of every box in the Young diagram.}
\label{hook}
\begin{picture}(22,12)

\end{picture}
\end{figure}

\subsection{The Burau representation}

For $n > 2$ the reduced Burau representation $\beta_{t}: B_n \rightarrow \mathrm{GL}_{n-1}(\mathbb{Z}[t^{\pm 1}])$ is determined by the matrices

\[ \beta_{t}(\sigma_1) = \left( \begin{array}{ccc}
-t & 1\\
0 & 1\\
\end{array} \right) \oplus I_{n-3}, \quad \beta_{t}(\sigma_{n-1}) = I_{n-3} \oplus \left( \begin{array}{ccc}
1 & 0\\
t & -t\\
\end{array} \right) \]
and for $1 < i < n-1$

\[ \beta_{t}(\sigma_{i}) = I_{i-2} \oplus \left( \begin{array}{ccc}
1 & 0 & 0\\
t & -t & 1\\
0 & 0 & 1\\
\end{array} \right) \oplus I_{n-i-2}. \]

A direct calculation shows that $(-\beta_t(\sigma_i))^2 = (t-1)(-\beta_t(\sigma_i))+t$. This relation was first observed by Jones. For $q^2=t$, the matrices $(-q^{-1}\beta_{q^2}(\sigma_i))$ satisfy the quadratic relation of $H(q,n)$. Since the Burau representation is irreducible and satisfies the quadratic relation, one can check that $(-q^{-1}\beta_{q^2}(\sigma_i)) = \pi_{\lambda}(\sigma_i)$, where $\lambda$ is the Young diagram {\tiny\yng(5,1)}.

\subsection{W-graphs}

In this section we provide a tool to explicit construct matrices of representations of the Hecke algebra $H(q,n)$. In fact, we will define a $\mathbb{Z}[q^{\pm 1}]$-module $E$, and an action of $H(q,n)$ on $E$. A W-graph encodes the structure of a $\mathbb{Z}[q^{\pm 1}]$-module in the sense that its vertices correspond to a basis for $E$; the edges provide all the information we need for the action of $H(q,n)$ on $E$. First we will give a new basis for $H(q,n)$; as before each basis element is associated to an element of the symmetric group $S_n$. Then we will define an equivalence relation on elements of $S_n$. The equivalence classes will be called cells. Vertices of a W-graph correspond to elements of a fixed cell. We will define an action of generators of $H(q,n)$ on the basis of $E$. This action extends to a representation $H(q,n)\rightarrow \mathrm{End}(E)$ \cite{KL}.

\paragraph{Definition of W-graphs.} Let $(S_n,S)$ be the Coxeter system for the symmetric group introduced in the previous section. A \emph{W-graph} is defined to be a set of vertices $X$ and a set of edges $Y$ together with the following data. For each vertex $x \in X$, we are given a subset $I_x \subset S$, and for each ordered pair of vertices $(y,w)$ with $\{y,w \} \in Y$, we are given an integer $\mu(y,w)$, subject to the requirements in the following paragraph.

Let $E$ be the free $\mathbb{Z}[q^{\pm 1}]$-module with basis associated to the vertex set $X$. Recall that $s_i$ is the transposition $(i,i+1)$ in $S_n$. For any $s_i \in S$, for any $w \in X$ (considering $X$ as a basis for $E$) we define the map $\tau_{s_i}$ as follows:
\[
       \tau_{s_i} w = 
  \begin{cases} 
      \hfill  -q^{-1}w, \hfill & \text{ if $s_i \in I_w$}, \\
      \hfill  qw+\sum \mu(y,w)y, \hfill & \text{ if $s_i \notin I_w$},\\
  \end{cases}
\]
where the sum is taken over all $y \in X,s_i \in I_y$ such that  $\{ y,w\} \in Y$. Extending linearly we get an endomorphism of $E$. For $i<n-2$, we require that
\[ \tau_{s_i} \tau_{s_{i+1}} \tau_{s_i} = \tau_{s_{i+1}} \tau_{s_i} \tau_{s_{i+1}}\]
and for $|i - j|>1$
\[ \tau_{s_i} \tau_{s_{j}} = \tau_{s_{j}} \tau_{s_i}.\]
In other words, we require that the endomorphisms $\tau_{s_i}$ satisfy the usual relations in the braid group.

We note that Kazhdan-Lusztig defined W-graphs for any Coxeter system.

\paragraph{Example 1.} We give an example of a W-graph for the group $S_3$ in Figure \ref{wgraphex}. We label the vertices of the W-graph by the elements of the set $X= \{s_1,s_2s_1\}$. We let $I_{s_1}=\{s_1\}$, $I_{s_2s_1}=\{s_2\}$ and the integers $\mu(s_1,s_2s_1)=1$, $\mu(s_2s_1,s_1)=0$. Furthermore, we have
\begin{eqnarray*}
\begin{tabular}{llll}
$\tau_{s_1}(s_1)$ & = &  $-q^{-1}s_1 $ \cr
$\tau_{s_1}(s_2s_1)$ & = &  $qs_2s_1+s_1 $  \cr
$\tau_{s_2}(s_1)$ & = &  $q s_1+s_2s_1 $  \cr
$\tau_{s_2}(s_2s_1)$ & = &  $-q^{-1} s_2s_1 $  \cr
\end {tabular}
\end {eqnarray*}
The matrices of the representation are as follows:

\[ 
T_{s_1} \mapsto \left( \begin{array}{ccccc}
-q^{-1} & 1\\
0 & q\\ \end{array} \right), \,  T_{s_2} \mapsto \left( \begin{array}{ccccc}
q & 0\\
1 & -q^{-1}\\ \end{array} \right).
\]
We note that in the literature the vertices of W-graphs are often labeled by $I_w$ where $w \in S_n$. 

\begin{figure}[h]
\begin{center}
\includegraphics[scale=.34]{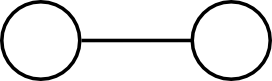}
\end{center}
\caption{Example of a W-graph}
\label{wgraphex}
\label{symbp}
\begin{picture}(22,12)
\put(180,60){$s_1$}
\put(228,60){$s_2s_1$}
\end{picture}
\end{figure}
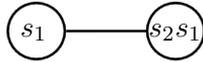

The endomorphism $\tau_{s_i}$ satisfies the quadratic relation $(\tau_{s_i}+q^{-1})(\tau_{s_i}-q)=0$. To see this consider $w \in S_n$ such that $s_i \in I_w$. Then we have:
$$(\tau_{s_i}+q^{-1})(\tau_{s_i}-q)(w)=(\tau_{s_i}+q^{-1})(-q^{-1}w-qw)=q^{-2}w+w-q^{-2}w-w=0.$$
On the other hand, if $s_i \notin I_w$, then we have:
\begin{eqnarray*}
\begin{tabular}{llll}
$(\tau_{s_i}+q^{-1})(\tau_{s_i}-q)(w)$ & = &  $ (\tau_{s_i}+q^{-1})(qw+ \sum \mu(y,w)y - qw) $ \cr
& = &  $\sum \mu(y,w)\tau_{s_i}(y) + q^{-1}\sum \mu(y,w)y$  \cr
& = &  $-q^{-1}\sum(y,w)y + q^{-1}\sum \mu(y,w)y$  \cr
& = &  $0$  \cr
\end {tabular}
\end {eqnarray*}

By the second condition of the definition of W-graphs, and since the endomorphism $\tau_{s_i}$ satisfies the quadratic relation, the map $T_{s_i} \mapsto \tau_{s_i}$ extends to a representation of the Hecke algebra $H(q,n)$.

\paragraph{The Kazhdan-Lusztig basis.} Let $a \rightarrow \overline{a}$ be the involution on the ring $\mathbb{Z}[q^{\pm 1}]$ defined by $\overline{q} = q^{-1}$. We extend this to an involution on $H(q,n)$ by the formula
\[ \overline{\sum a_w T_w} = \sum \overline{a} T^{-1}_{w^{-1}}. \]
We consider the set of generators $S$ of $S_n$, and the length function $l:S_n \rightarrow \mathbb{Z}$ described in Section 2. We denote by $\leqq$ the Bruhat order on the set of all words in the generating set $S$, that is, $y \leqq x$ if $y$ is subword of $x$.

We have the following theorem of Kazhdan-Lusztig \cite[Theorem 1]{KL} reformulated by Yin \cite[Theorem 1.4]{YI}:

\begin{theorem}
\label{K-L}
For any $w \in S_n$, there is a unique element $C_w \in H(q,n)$ such that $\overline{C_w} = C_w$, where
\[ C_w = \sum_{y \leqq w} (-q)^{l(w)-l(y)} \overline{P_{y,w}(q^2)} T_y \]
and $P_{y,w}$ is a polynomial of degree at most $\frac{1}{2} (l(w)-l(y)-1)$ if $y<w$, and $P_{w,w}=1$.
\end{theorem}

The polynomial $P_{y,w}$ is known in the literature as the \emph{Kazhdan-Lusztig polynomial}, and Theorem \ref{K-L} proves the existence and uniqueness of $P_{y,w}$. The Kazhdan-Lusztig polynomial is difficult to construct explicitly and we will not give details of the construction in here. A recursive formula for the computation of the Kazhdan-Lusztig polynomial can be found in the original paper \cite[Equation 2.2c]{KL}.

We will show by induction, on $l(w)$, that $T_w$ can be expressed as a linear combination of elements of $\{ C_x \mid x \in S_n \}$. We have that $C_1=T_1$ and $C_{s_i}=T_{s_i}-qT_{1}$, where $1$ stands for the trivial element of $S_n$, and $s_i$ is the transposition $(i,i+1)$ in $S$. For an arbitrary $w \in S_n$ we have
\[ C_w  = (\sum_{y < w} (-q)^{l(w)-l(y)} \overline{P_{y,w}(q^2)} T_y) + T_w. \]
Since $l(y)<l(w)$, by the inductive hypothesis we have that $T_y$ can be expressed as a linear combination of $C_x$, for elements $x \in S_n$ with $l(x)<l(w)$. Furthermore, since the cardinality of $\{ C_x \mid x \in S_n \}$ is equal to the cardinality of $\{ T_x \mid x \in S_n \}$, then the set $\{ C_x \mid x \in S_n \}$ forms a basis for $H(q,n)$.

\paragraph{Construction of W-graphs and cells.} Consider $y,w \in S_n$. We say that $y \prec w$ if $y<w$, if $(-1)^{l(y)}=-(-1)^{l(w)}$, and if the Kazhdan-Lusztig polynomial $P_{y,w}$ has degree exactly $\frac{1}{2} (l(w)-l(y)-1)$. We take $\mu(y,w)$ to be the coefficient of the highest power of $q$ in $P_{y,w}$. Let $\Gamma$ be the graph whose vertices $X$ correspond to the $n!$ elements of $S_n$ and whose edges are subsets of $S_n$  of the form $\{ y,w \}$ with $y \prec w$. We set $I_w = \{ s_i \in S \mid s_iw<w \}$. We define a preorder relation $\leqq_{\Gamma}$ on the set of vertices of $\Gamma$ as follows. Two vertices $x,x'$ satisfy $x \leqq_{\Gamma} x'$ if there exist a sequence of vertices $x=x_0,x_1,...,x_n=x'$ such that for each $i$, $(1\leq i \leq n)$, $\{x_{i-1},x_i \}$ is an edge and $I_{x_{i-1}} \nsubseteq I_{x_i}$. Define the equivalence relation $x \sim_{\Gamma} x'$ if $x \leqq_{\Gamma} x' \leqq_{\Gamma} x$. The equivalence classes of $S_n$ under $\sim_{\Gamma}$, denoted by $[w]$, are called \emph{cells}. We denote by $\Gamma_{[w]}$ the subgraph of $\Gamma$ whose vertices correspond to elements of $[w]$. Now we can define the action of $T_{s_i}$ on $\{C_{w} \mid w \in S_n \}$:

\[ T_{s_i} C_w = q C_w + C_{s_i w} + \sum_{y \prec w, s_i y < y} \mu(y,w) C_y. \]
This action above is well defined \cite[Theorem 2.5]{YI}.

Let $D_{w}$ be the $\mathbb{Z}[q^{\pm 1}]$-module spanned by the set $\{ C_y \mid y \leqq_{\Gamma} w \}$, and let $D'_{w}$ be the $\mathbb{Z}[q^{\pm 1}]$-module spanned by the set $\{ C_y \mid y \leqq_{\Gamma} w, y \notin [w] \}$. It is obvious that $D'_{w}$ is contained in $D_{w}$. We will show that $D_{w}$ and $D'_{w}$ are left-ideals of $H(q,n)$. Recall the formula
\[ T_{s_i} C_w = q C_w + C_{s_i w} + \sum_{y \prec w, s_i y < y} \mu(y,w) C_y. \]
The conditions $y \prec w, s_i y < y$ together with the existence of $\mu(y,w)$ show that $y\leqq_{\Gamma} w $. Furthermore $s_i w \leqq_{\Gamma} w$ since $w \prec s_i w$ (more particularly we have $P_{w,s_i w}=1$ \cite[Lemma 2.6 (iii)]{KL}), and $I_{w} \nsubseteq I_{s_i w}$. Hence we can define the quotient $D_w / D'_w$. We have the following theorem \cite[Theorem 2.6]{YI}:

\begin{theorem}
\label{Wrep}
The graph $\Gamma_{[w]}$ defined above is a W-graph whose associated $\mathbb{Z}[q^{\pm 1}]$-module is $D_{w}/D'_{w}$.
\end{theorem}

\begin{flushleft}
By the theorem above we have a well defined representation
\end{flushleft}
\[ \rho : H(q,n) \rightarrow \mathrm{End}(D_{w}/D'_{w}) \]
where the elements in the basis of $D_{w}/D'_{w}$ correspond to the elements of the cell $[w]$. Kazhdan-Lusztig proved that the representation of $H(q,n)$ arising from the action on cells is irreducible and that the isomorphism class of the W-graph depends only on $\rho$ and not on $[w]$ \cite[Theorem 1.4]{KL}.

\paragraph{Example of Figure \ref{wgraphex}.} We have already seen that $s_2 s_1 \leqq_{\Gamma} s_1 $. By denoting the trivial element of $S_n$ by 1, we have that $1 \prec s_1$ with $s_1 \leqq_{\Gamma} 1$. Then we have $1 \prec s_1 s_2 s_1$ with $1 \leqq_{\Gamma} s_1 s_2 s_1$. Finally $s_1 s_2 s_1 \leqq_{\Gamma} s_2 s_1$. Hence $s_1 \sim_{\Gamma} s_2 s_1$; the vertex set $\{ s_1, s_2 s_1 \}$ is labeled by $\mu(s_1,s_2 s_1)=1$ \cite[Theorem 2.6]{KL}.

\subsection{Robinson-Schensted correspondence and dual Knuth equivalence}

In this section we will give an algorithm for finding elements of cells defined in the previous section. Unfortunately the construction of cells given by Kazhdan-Lusztig \cite{KL} is not easy when $n$ is large. For that reason we will provide a different strategy to obtain cells. We have two tools: the Robinson-Schensted correspondence (RS-correspondence), and the dual Knuth equivalence. Using the RS-correspondence, we will be able to decide when two elements of $S_n$ belong to the same cell. We will also be able to calculate the cardinality of a cell, that is, the dimension of the representation associated to that cell. By the dual Knuth equivalence, given a element $w \in S_n$ we will be able to obtain all elements of $[w]$. That is, by taking any $w\in S_n$, we will find all $y \in S_n$ such that $w \sim_{\Gamma} y$.

\paragraph{Robinson-Schensted correspondence.} Here we will give an algorithm that associates a Young diagram to a given element of $S_n$. More precisely, in this algorithm every element of $S_n$ corresponds to two standard tableaux of the same shape. For $w \in S_n$, let $w_i$ denote the image of $w(i)$ under the mapping
$$w: \{1,2,...,n \} \rightarrow \{1,2,...,n \} .$$
We will identify $w$ with the sequence $w_1 w_2 ... w_n$. Consider an arbitrary tableaux $T$. We denote the boxes of the i$^{th}$ row by $R_i(T)$. In the first step the row $R_1(T)$ contains a box filled by $w_1$. In the next steps if $w_j$ is greater than every number in $R_i(T)$ then we add a box on the right of all other boxes filled by $w_j$. If $w_j$ is not greater than every number of $R_i(T)$, we consider $k \in R_i(T)$ such that $k$ is the smallest number for which $w_j<k$. We replace $k$ by $w_j$. If $R_{i+1}(T)$ does not exist, we add a box in $R_{i+1}(T)$ filled by $k$. If not, we repeat the same process with $k$ playing the role of $w_j$ in the $R_{i+1}(T)$ row. The algorithm ends when we insert all $w_j$ in the boxes of $T$. The algorithm we described is called \emph{row insertion algorithm}. The standard diagram we obtain is denoted by $P(w)$, and it is called \emph{P-symbol}. We define the \emph{Q-symbol} to be $Q(w)=P(w^{-1})$. For example in $S_3$ consider the element $s_1 s_2 = 231$, where $s_i \in S$ are transpositions. In Figure \ref{symbpp} we compute the P-symbol of $s_1 s_2$ step by step.

\begin{figure}[h]
\begin{center}
\includegraphics[scale=.5]{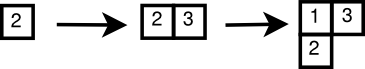}
\end{center}
\caption{The P-symbol of $s_1 s_2$.}
\label{symbpp}
\begin{picture}(22,12)

\end{picture}
\end{figure}

The pair $(P(w),Q(w))$ is called \emph{Robinson-Schensted correspondence}. There is a bijection between the elements of $S_n$ and the pairs $(P(w),Q(w))$ for all $w\in S_n$. For proof of Theorem \ref{RS-eq} see \cite[Theorem A]{AS}. 

\begin{theorem}
For $y,w \in S_n$ we have $y \sim_{\Gamma} w$ if and only if $Q(y) = Q(w)$.
\label{RS-eq}
\end{theorem}

As we have seen in the previous section Kashdan-Luzstig proved that the action of $H(q,n)$ on cells induces irreducible representations. But the connection of cells and Young diagrams is lacking. The latter connection is summarized in the following theorem \cite[Theorems 6.5.2, 6.5.3]{BB1}. For a fixed $x \in S_n$, we denote the shape of $Q(x)$ by $S(Q(x))$.

\begin{theorem}
\label{connex}
The irreducible representation of $H(q,n)$ associated with the cell $[w]$, for a fixed $w \in S_n$, is labeled by a Young diagram $\lambda$, where $\lambda = S(Q(w))$.
\end{theorem}

\paragraph{Dual Knuth equivalence.} Here we provide an algorithm to find all elements of a cell. That is, starting with an element of $S_n$, we have a process that will allow us to obtain elements of the cell associated to a given element of $S_n$ \cite[Appendix A3.6]{BB1}.

For $x,y \in S_n$, we write $x \sim_{dK} y$ ($x$ is dual Knuth equivalent to $y$) if $x$ and $y$ differ by transposition of two values $i$ and $i+1$, and either $i-1$ or $i+2$ occurs in a position between those of $i$ and $i+1$. For example
\[ 215436 \sim_{dK} 315426 \sim_{dK} 415326 \sim_{dK} 425316  \]
shows that $215436 \sim_{dK} 425316$. We have the next Theorem \cite[Fact A3.6.2 ]{BB1}.

\begin{theorem}
For $x,y \in S_n$, we have that $Q(x) = Q(y)$ if and only if $x \sim_{dK} y$.
\end{theorem}

The algorithm for construction of a cell is divided into 2 steps.

\begin{itemize}
\item[1.] Fix an element $w$ of $S_n$. By Theorem \ref{connex} we obtain the cardinality of $[w]$.
\item[2.] By the dual Knuth equivalence, we obtain all elements of $[w]$. For any $x \in [w]$ we deduce $I_x=\{ s_i \in S \mid s_ix<x \}$.
\end{itemize}

In the previous section we defined a class of $\mathbb{Z}[q^{\pm 1}]$-modules and an action of the Hecke algebras on these modules. Since the construction of these modules by definition is difficult, we can use the algorithm described above to construct them.

\subsection{Facts about the Hecke algebra representation}
We fix an element $w\in S_n$. Let $ H(q,n) \rightarrow \mathrm{End}(V)$ be a Hecke algebra representation where $V$ is the $\mathbb{Z}[q^{\pm 1}]$-module spanned by $C_x$ for all $x \in [w]$ as described in Section 3.2. Since $H(q,n)$ is a quotient of the group algebra of $B_n$, there is a well defined representation $\pi_{\lambda} :B_n \rightarrow \mathrm{End}(V)$. Furthermore, $\lambda$ is the shape of the standard diagram of $Q(w)$.

\paragraph{Representations when the Young diagram is rectangular.} We now focus on representations of $B_{2g+2}$ when $g \geq 2$ and $\lambda$ is a rectangular diagram. For $g\geq 2$ we want to construct a representation $\pi_{\lambda}: B_{2g+2} \rightarrow \mathrm{End}(V_{\lambda})$ such that the matrices in the image of $\pi_{\lambda}$ follow a pattern when the $g$ increases. First we want to find cells that correspond to rectangular diagrams. There is a difficulty here. Consider for example $w \in S_6$ such that $[w]$ is a cell associated to a rectangular diagram. Furthermore $w$ is an element of $S_8$. But by the RS-correspondence we can see that $[w]$ in $S_8$ does not always correspond to a rectangular diagram.\\

We can solve this problem by making good choices for cells. We denote the generators of $S_{2g+2}$ by the transpositions $s_i$, $i <2g+2$, and we consider the element $s_1 s_3 s_5 ... s_{2g+1} \in S_{2g+2}$. Recall from Section 3.3 that for $x \in S_n$, $Q(x)$ is the Q-symbol. By the RS-correspondence it is easy to check that the shape of $Q(s_1 s_3 s_5 ... s_{2g+1})$, denoted by $\lambda_{2g+2}$, is rectangle as indicated in Figure \ref{rect} for $g=5$.
\begin{figure}[h]
\begin{center}
\includegraphics[scale=.5]{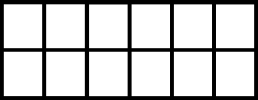}
\end{center}
\caption{The Young diagram of the cell $[s_1 s_3 s_5 ... s_{11}]$.}
\label{rect}
\begin{picture}(22,12)

\end{picture}
\end{figure}

For $g\geq2$, we consider the cells $[s_1 s_3 s_5] = W_1$, $[s_1 s_3 s_5 ... s_{2g+1}] = W_2$. The cell $[s_1 s_3 s_5]$ contains the elements 
\[\{ s_1 s_3 s_5, s_1 s_4 s_3 s_5, s_2 s_1 s_3 s_5, s_2 s_1 s_4 s_3s_5,s_3s_2s_1s_4s_3s_5 \}.\]
For $g=3$ the cell $[s_1 s_3 s_5s_7]$ contains the elements
\[\{s_1s_3s_5s_7,s_1s_4s_3s_5s_7,s_2s_1s_3s_5s_7,s_2s_1s_4s_3s_5s_7,s_3s_2s_1s_4s_3s_5s_7,\]
\[s_1s_3s_6s_5s_7,s_1s_5s_4s_3s_6s_5s_7,s_2s_1s_4s_3s_6s_5s_7,s_2s_1s_5s_4s_3s_6s_5s_7,,s_2s_1s_3s_6s_5s_7\]
\[s_1s_4s_3s_6s_5s_7,s_3s_2s_1s_4s_3s_6s_5s_7,s_3s_2s_1s_5s_4s_3s_6s_5s_7,s_4s_3s_2s_1s_5s_4s_3s_6s_5s_7\}. \]
The first five elements of $[s_1 s_3 s_5 s_7]$ above differ from the elements of $[s_1 s_3 s_5]$ by the generator $s_7$. For $g>3$ the first five elements of $[s_1 s_3 s_5 ... s_{2g+1}]$ differ from $[s_1 s_3 s_5]$ by the word $s_7 s_9 ... s_{2g+1}$.\\

Recall from Section 2 the correspondence $\sigma_i \mapsto T_{s_i}$ where $T_{s_i}$ is a generator of $H(q,n)$. We have the following theorem.

\begin{theorem}
\label{main3}
The map $C_w \mapsto C_{w s_7s_9...s_{2g+1}}$ for $w\in W_1$, and $w s_7s_9...s_{2g+1} \in W_2$ defines an embedding for $H(q,6)$-modules $V_{\lambda_6} \hookrightarrow V_{\lambda_{2g+2}}|_{H(q,6)}$.
\end{theorem}

In other words we get a representation $B_{2g+2} \rightarrow \mathrm{End}(V_{\lambda_{2g+2}})$, such that
\[ \sigma_i \mapsto \pi_{\lambda_{2g+2}}(\sigma_i) =  \left( \begin{array}{ccc}
\pi_{\lambda_{6}}(\sigma_i) & A \\
0 & B \end{array} \right)\]
where $\pi_{\lambda_{6}}(\sigma_i)$ is a $5 \times 5$ matrix, $d$ is the dimension of $\pi_{\lambda_{2g+2}}(\sigma_i)$, $A$ is a matrix of dimension $5 \times (d -5)$, and $B$ is a matrix of dimension $(d-5) \times (d-5) $.

\begin{proof}
We denote the elements of $W_1$ by $u_1,u_2,u_3,u_4,u_5$. Then the first five elements of $W_2$ have the form $u_i s_7 s_9 ... s_{2g+1} = w_i$, where $i \leq 5$. Recall from Section 3.2 that the action of $T_{s_j}$ on $C_{w_i}$ is defined as follows:
\[ T_{s_j} C_{w_i} = q C_{w_i} + C_{s_j w_i} + \sum_{y \prec w_i, s_j y < y} \mu(y,w_i) C_{y}. \]
We divide the proof into two steps. In the first step we show that the basis elements $C_{x}$, $x \in W_2$ in the above sum are the same for any $g \geq2$. In the second step we show that $\mu(y',w_i)=\mu(y,u_i)$ for $y' \prec w_i, s_j y' < y'$, $y \prec u_i, s_j y < y$, and $y,y'$ differ by the word $s_7s_9...s_{2g+1}$. 

\paragraph{Step 1.} By Theorem \ref{Wrep} if $s_j w_i<w_i$, then $T_{s_j} C_{w_i} = -q^{-1}C_{w_i}$; if $s_j w_i>w_i$ then the element $C_{s_j w_i}$ vanishes in $V_{\lambda_{2g+2}}$ (the $\mathbb{Z}[q^{\pm 1}]$-module spanned by $C_x$ for all $x \in W_2$).

We will show that if $y \prec w_i, s_j y < y$, then $y$ is one of the first five elements in $W_2$. By considering a weaker restriction $y\leqq w_i$, we will show that either $y \notin W_2$ or $y$ is one of $w_i$. We can see that the argument of the first step is true for $g=2$, and 3 (see the examples of cells above). If $g\geq4$ we note that $u_i$ differ from $w_i$ by the word $s_7s_9...s_{2g+1}$, and every $u_i$ is a word in $\{ s_j \mid j \leq5 \}$. Therefore, by the RS-correspondence for $y \in S_{2g+2}$ such that $y \leqq w_i$ we have that either $y \notin W_2$ or $y$ is one of $w_i$.

\paragraph{Step 2.} If $k = 1$, and $i = 5$ do not occur simultaneously, then $\mu(w_k,w_i)=\mu(u_k,u_i)=1$ since $l(w_i)-l(w_k) \leq 2$ \cite[Theorem 2.6]{KL}. Where $l:S_{2g+2} \rightarrow \mathbb{Z}$ is the length function defined in Section 2. It remains to prove that $\mu(w_1,w_5)=\mu(u_1,u_5)=1$. We will compute the polynomial $P_{w_1,w_5}$. We know that $deg(P_{w_1,w_5})=\frac{l(w_2)-l(w_1)-1}{2}=1$, hence $P_{w_1,w_2}=\mu(w_1,w_2)q+1$. By the recursive formula given by Kazhdan-Lusztig \cite[2.2c]{KL} we have
\[ P_{w_1,w_5}= P_{s_3 w_1,s_3 w_5} + qP_{w_1,s_3 w_5} - \sum \mu(z,s_3 w_5)q P_{w_1,z} \]
where the sum is over all $z$ such that $w_1 \leqq z \prec s_3 w_5, s_3 z < z$. But since we have $P_{w_1,s_3 w_5}=\mu(z,s_3 w_5)= P_{w_1,z}=1$, \cite[theorem 2.6]{KL} we conclude
\[ P_{w_1,w_5}= P_{s_3 w_1,s_3 w_5} + q - \sum q . \]
Furthermore, by the RS-correspondence we can easily check that the elements $s_3 w_1$ and $s_3 w_5$ are not equivalent, which implies that the degree of $P_{s_3 w_1,s_3 w_5}$ is strictly less than 1. Hence $P_{s_3 w_1,s_3 w_5}=1$. Finally, the kazhdan-lusztig polynomials have nonnegative integer coefficients \cite[Corollary 1.2]{KL1}. Therefore $P_{w_1,w_5}=q+1$, and we have deduced that $\mu(w_1,w_5)=\mu(u_1,u_5)=1$.
\end{proof}

%\paragraph{Note.} For a rectangular Young diagram $\lambda_{2g+2}$ as in Theorem \ref{main3}, the dimension of the associated representation is equal to 
%\[ \binom{2g+2}{g+1} \frac{1}{g+2}. \]
%Furthermore, according to Lemma \ref{jone1} the rank of the idempotent $e_i$ is equal to the descending paths from {\tiny\yng(1,1)} to $\lambda_{2g+2}$ of the Young's lattice. But the latter is equal to the number of the descending paths from $\APLbox$ to $\lambda_{2g}$, which is equal to
%\[ \binom{2g}{g} \frac{1}{g+1}.  \]

So far we have constructed irreducible representations of the braid groups $B_n$ that factor through Hecke algebras $H(q,n)$. Theorem \ref{main3} is the key that we need to prove the Theorem A in Section 5. %Furthermore, we calculated the image of the center of $B_n$; we will use this result in the next section to construct representations for the hyperelliptic mapping class group. Finally, Theorem \ref{main3} is the key that we need to prove the Theorem A in Section 5.

\section{Representation for the mapping class group a punctured sphere}

%In this section we use the theory we described in Section 3 about braid group representations, to define representations for the hyperelliptic mapping class group.

For $i=1,2,...,2g+1$, let $\sigma_i,H_i$ be the generators of $B_{2g+2}$, $\mathrm{Mod}(\Sigma_{0,2g+2})$ respectively as in Section 2. Since the homomorphism $B_{2g+2} \rightarrow \mathrm{Mod}(\Sigma_{0,2g+2})$ is surjective, it is reasonable to ask whether we can define a linear representation of $\mathrm{Mod}(\Sigma_{0,2g+2})$ via $H(q,2g+2)$. An affirmative answer was given by Jones \cite[Theorem 10.2]{JO}. Since we use a different definition of $H(q,n)$ than Jones, we will reformulate his representation. But first we need to prove two lemmas.

We denote by $f_i$ the image of $\sigma_i$ in $\mathrm{End}(V)$, and for $q^2\neq -1$ we set $e_i = (q-f_i)/(q+q^{-1})$. In this Section we examine certain properties of the representation of the braid group $\pi_{\lambda} :B_n \rightarrow \mathrm{End}(V)$. Firstly, we will calculate the image of the center of $B_n$ under the map  $\pi_{\lambda}$ (Lemma \ref{LongEl}).

\begin{lemma}
We have $e^2_i=e_i$. The rank of the idempotent $e_i$ is the number of descending paths from diagram $\lambda_0=$ {\tiny\yng(1,1)} to diagram $\lambda$ of Young's lattice described in Section 3.1.
\label{jone1}
\end{lemma}

\begin{proof}
We can easily check that $e^2_i=e_i$. The rank of an idempotent in $\mathrm{End}(V)$ is equal to its trace. Since all generators of $B_n$ are conjugate in $B_n$, then all $e_i$ have the same rank (trace). It suffice to calculate the rank of $e_1$. We have that $\pi_{\lambda_0}(\sigma_1)=-q^{-1}$. Thus $\pi_{\lambda_0}(e_1)=1$. We will prove the lemma by induction. We restrict the representation of $B_n$ to $B_{n-1}$. Then the image of $\sigma_1$ in $\mathrm{End}(V)$ is $\bigoplus \pi_{\lambda_i}(\sigma_1)$. By the induction argument the rank of $\pi_{\lambda_i}(\sigma_1)$ is the number of paths from $\lambda_0$ to $\lambda_i$. Since the rank of $\pi_{\lambda}(\sigma_1)$ is the sum of ranks f $\pi_{\lambda_i}(\sigma_1)$, we conclude that the rank is equal to the number of paths from the diagram $\lambda_0$ to the diagram $\lambda$. 
\end{proof}

\begin{lemma}
If $\mathrm{dim}(\pi_{\lambda})=d$ and $\mathrm{rank}(e_i)=r$ then
\[ \pi_{\lambda}((\sigma_1 \sigma_2 ... \sigma_{n-1})^n) = q^{n(n-1)\frac{d-2r}{d}} Id_{\pi_{\lambda}}. \]
\label{LongEl}
\end{lemma}

\begin{proof}
The element $(\sigma_1 \sigma_2 ... \sigma_{n-1})^n$ is in the center of $B_n$; thus $\pi_{\lambda}((\sigma_1 \sigma_2 ... \sigma_{n-1})^n)$ is a diagonal matrix whose entries are all equal. Since 
$$\mathrm{det}(\pi_{\lambda}(\sigma_i)) = \mathrm{det}(f_i) = \mathrm{det}(q-qe_i-q^{-1}e_i) = \mathrm{det}(q) \mathrm{det}(1-e_i-q^{-2}e_i) = (-1)^r q^{d-2r},$$
we have that $\mathrm{det}(\pi_{\lambda}((\sigma_1 \sigma_2 ... \sigma_{n-1})^n)) = q^{n(n-1)(d-2r)}$. Hence
\[ \pi_{\lambda}((\sigma_1 \sigma_2 ... \sigma_{n-1})^n) = \omega q^{n(n-1)\frac{d-2r}{d}} \]
for some $d^{th}$ root of unity $\omega$. But $\omega$ depends continuously on $q$. But if we put $q=1$, then we obtain a representation for the symmetric group $S_n$. But the center of $S_n$ is trivial. Hence $\omega = Id_{\pi_{\lambda}}$. 
\end{proof}

\paragraph{Note.} For a rectangular Young diagram $\lambda_{2g+2}$ as in Theorem \ref{main3}, the dimension of the associated representation is equal to 
\[ \binom{2g+2}{g+1} \frac{1}{g+2}. \]
Furthermore, according to Lemma \ref{jone1} the rank of the idempotent $e_i$ is equal to the descending paths from {\tiny\yng(1,1)} to $\lambda_{2g+2}$ of the Young's lattice. But the latter is equal to the number of the descending paths from $\APLbox$ to $\lambda_{2g}$, which is equal to
\[ \binom{2g}{g} \frac{1}{g+1}.  \]

Now we will use Lemmas \ref{jone1} and \ref{LongEl} to construct a representation for $\mathrm{Mod}(\Sigma_{0,2g+2})$.

\begin{proposition}
\label{BrRep}
Consider the representation $\pi_{\lambda}: B_{2g+2} \rightarrow \mathrm{End}(V_{\lambda})$ associated to the Young diagram $\lambda$. We set $\pi'_{\lambda}(\sigma_i) = q^{(2r-d)/d} \pi_{\lambda}(\sigma_i)$. Then the map 
\[J: \mathrm{Mod}(\Sigma_{0,2g+2}) \rightarrow \mathrm{End}(V_{\lambda})\]
defines a representation via $J(H_i) = \pi'_{\lambda}(\sigma_i)$ if and only if $\lambda$ is rectangular.
\end{proposition}

We note that the homomorphism $J: \mathrm{Mod}(\Sigma_{0,2g+2}) \rightarrow \mathrm{End}(V_{\lambda})$ is called the Jones representation.

\begin{center}
\begin{tikzpicture}[descr/.style={fill=white,inner sep=2.5pt}]
\matrix (m) [matrix of math nodes, row sep=3em,
column sep=3em]
{ B_{2g+2} & \mathrm{GL_{d}(\mathbb{Z}[q^{\pm 1}])} \\
\mathrm{Mod}(\Sigma_{0,2g+2}) \\ };
\path[->,font=\scriptsize]
(m-1-1) edge node[auto] {} (m-2-1)
edge node[auto] {$\pi_{\lambda}$} (m-1-2)
(m-2-1) edge node[auto] {$\pi'_{\lambda}$} (m-1-2);
\end{tikzpicture}
\end{center}

\begin{proof}
We will show that the elements $\pi'_{\lambda}(\sigma_i)$ satisfy the relations of $\mathrm{Mod}(\Sigma_{0,2g+2})$ defined in Section 2. First we assume that $\lambda$ is rectangular. The braid relation and the disjointness relation are satisfied for $\pi'_{\lambda}(\sigma_i)$. By Lemma \ref{LongEl} we have that $\pi'_{\lambda}((\sigma_1 \sigma_2...\sigma_{2g+1})^{2g+2})$ is trivial. We note that in the braid group we have the relation 
$$\sigma_1 \sigma_2 ... \sigma^2_{2g+1} \sigma_{2g}... \sigma_1 = (\sigma_1 \sigma_2...\sigma_{2g+1})^{2g+2} (\sigma_2...\sigma_{2g+1})^{-(2g+1)}. $$
Then, the condition $\pi'_{\lambda}( \sigma_1 \sigma_2 ... \sigma^2_{2g+1} \sigma_{2g}... \sigma_1)=1$ is equivalent to $\pi'_{\lambda}( (\sigma_2...\sigma_{2g+1})^{2g+1})=1$, since we already have that $\pi'_{\lambda}((\sigma_1 \sigma_2...\sigma_{2g+1})^{2g+2})$ is trivial. But the restriction $\pi'_{\lambda} |_{B_{2g+1}}$ when $\lambda$ is rectangular, satisfies the relation $\pi'_{\lambda}(\sigma_2...\sigma_{2g+1})^{2g+1}=1$. We note that by the Young's lattice, the dimension $d$ and the rank $r$ after the restriction above do not change. This proves the `if' part.

If $\lambda$ is not rectangular then $\pi'_{\lambda}$ restricted to $B_{2g+1}$ reduces as the direct sum of representations $\pi'_{\lambda_i}$. For $i\leq k$ and each $\pi'_{\lambda_i}$ we have the numbers $r_i$ and $d_i$ such that $d = \sum^{k}_{i=1}d_i$ and $r = \sum^{k}_{i=1}r_i$. But the only way to have $\pi'_{\lambda_i}(\sigma_2...\sigma_{2g+1})^{2g+1}=1$ for $(d_i-2r_i) / d_i$ is to be $(d-2r)/d$ for all $i$. But this is impossible if $k>1$. 
\end{proof}

Precomposing the map of Proposition \ref{BrRep} with the surjective homomorphism
$$\mathrm{SMod}(\Sigma_g) \rightarrow \mathrm{Mod}(\Sigma_{0,2g+2})$$
we obtain the following corollary.

\begin{corollary}
There is a well defined representation $ \mathrm{SMod}(\Sigma_g) \rightarrow \mathrm{End}(V_{\lambda})$ defined by $T_{c_i} \mapsto q^{(2r-d)/d} \pi_{\lambda}(\sigma_i)$ if and only if $\lambda$ is rectangular.
\end{corollary}

By using the formula given in the definition of W-graphs, the action of $H(q,6)$ on $[s_1 s_3 s_5]$ give the following matrices for $\mathrm{Mod}(\Sigma_{0,2g+2})$

\[ 
H_1 \mapsto q^{-1/5} \left( \begin{array}{ccccc}
-q^{-1} & 0 & 1 & 0 & 1\\
0 & -q^{-1} & 0 & 1 & 0\\
0 & 0 & q & 0 & 0\\
0 & 0 & 0 & q & 0\\
0 & 0 & 0 & 0 & q\\ \end{array} \right), \,  H_2 \mapsto q^{-1/5} \left( \begin{array}{ccccc}
q & 0 & 0 & 0 & 0\\
0 & q & 0 & 0 & 0\\
1 & 0 & -q^{-1} & 0 & 0\\
0 & 1 & 0 & -q^{-1} & 1\\
0 & 0 & 0 & 0 & q\\ \end{array} \right)
\]

\[ 
H_3 \mapsto q^{-1/5} \left( \begin{array}{ccccc}
-q^{-1} & 1 & 1 & 0 & 0\\
0 & q & 0 & 0 & 0\\
0 & 0 & q & 0 & 0\\
0 & 0 & 0 & q & 0\\
0 & 0 & 0 & 1 & -q^{-1}\\ \end{array} \right), \, H_4 \mapsto q^{-1/5} \left( \begin{array}{ccccc}
q & 0 & 0 & 0 & 0\\
1 & -q^{-1} & 0 & 0 & 0\\
0 & 0 & q & 0 & 0\\
0 & 0 & 1 & -q^{-1} & 1\\
0 & 0 & 0 & 0 & q\\ \end{array} \right)
\]

\[ H_5 \mapsto q^{-1/5} \left( \begin{array}{ccccc}
-q^{-1} & 1 & 0 & 0 & 1\\
0 & q & 0 & 0 & 0\\
0 & 0 & -q^{-1} & 1 & 0\\
0 & 0 & 0 & q & 0\\
0 & 0 & 0 & 0 & q\\ \end{array} \right) \]

In this section we constructed linear representations of the mapping class groups $\mathrm{Mod}(\Sigma_{0,2g+2})$.

\section{Proof of Theorem A}

\paragraph{Representation over the normal closure of a power of a half-twist.} Let $\mathcal{N}(h)$ denote the normal closure of $h$ in $\mathrm{Mod}(\Sigma_{0,2g+2})$. In this section we will use the Jones representation $J : \mathrm{Mod}(\Sigma_{0,2g+2}) \rightarrow \mathrm{End}(V_{\lambda})$ to construct a linear representation for $\mathrm{Mod}( \Sigma_{0,2g+2}) / \mathcal{N}(H^n_i)$. As noted in the introduction, Humphries constructed a linear representation for the group $\mathrm{Mod}( \Sigma_{0,6}) / \mathcal{N}(H^n_i)$, and proved that $\mathrm{Mod}( \Sigma_{0,2g+2}) / \mathcal{N}(H^n_i)$ has infinite order if $n \geq 4$ \cite[Theorem 4]{HU}. Our aim here is to extend Humphries' result for any $g \geq 2$.

Let $\lambda$ be the Young diagram associated to the cell $[s_1s_3s_5...s_{2g+1}]$ as before. Consider the representation $\pi_{\lambda}:B_{2g+2} \rightarrow \mathrm{End}(V_{\lambda})$. Recall from Section 4 that $\pi_{\lambda}(\sigma_i)=f_i$ and $e_i = (q-f_i)/(q+q^{-1})$ for $q^2\neq-1$. Then we have $J(H_i) = q^{(2r-d)/d} (q-(q+q^{-1})e_i)$. We want to compute $J(H^n_i)= q^{n(2r-d)/d} (q-(q+q^{-1})e_i)^n$. Our aim is to modify the representation $J : \mathrm{Mod}(\Sigma_{0,2g+2}) \rightarrow \mathrm{End}(V_{\lambda})$ such that the image of $H^n_i$ will be trivial, and the relations of $\mathrm{Mod}(\Sigma_{0,2g+2})$ will still hold. Then we will have a well defined linear representation for $\mathrm{Mod}(\Sigma_{0,2g+2}) / \mathcal{N}(H^n_i)$. By the binomial theorem, and the fact that $e^j_i=e_i$ for any $j\geq 2$ we have the following:

\begin{eqnarray*}
\begin{tabular}{llll}
$(q-(q+q^{-1})e_i)^n$ & = &  $ \sum^{n}_{j=0} (-1)^j {n \choose j} q^{n-j}(q+q^{-1})^j e^j_i $ \cr
& = &  $q^n + \sum^{n}_{j=1} (-1)^j {n \choose j} q^{n-j}(q+q^{-1})^j e_i$  \cr
& = &  $q^n + e_i (-q^n + e_i \sum^{n}_{j=0} (-1)^j {n \choose j} q^{n-j}(q+q^{-1})^j)$  \cr
& = &  $q^n + e_i ((q-q-q^{-1})^n -q^n)$  \cr
& = &  $q^n + e_i ((-1)^n q^{-n} -q^n)$ \cr
\end {tabular}
\end {eqnarray*}
Hence we have

\begin{equation}\label{equ}
J(H^n_{i})= q^{\frac{2nr}{d}} + e_i((-1)^n q^{2n\frac{r-d}{d}} -q^{\frac{2nr}{d}}).\tag{*}
\end{equation}

\paragraph{Case $\textbf{n}$ odd.} It is convenient to change $q^{2/d}$ to $t$ to obtain
\[ J(H^n_{i})= t^{nr} + e_i((-1)^n t^{n(r-d)} -t^{nr}).  \]
We let $(-1)^nt^d$ be an $n$ root of unity. Then we have $t^{nd} = -1$, and $t^n = (-1)^{1/d}$. If $J'(H_i) = (-1)^{-r/d}J(H_i)$, then $J'(H^n_{i})=1$. To see that the map
\[J' : \mathrm{Mod}(\Sigma_{0,2g+2}) / \mathcal{N}(H^n_{i}) \rightarrow \mathrm{GL}_d(\mathbb{C}).\]
is a homomorphism, we only need to check that $J'(H_{i})$ satisfy the relations of $\mathrm{Mod}(\Sigma_{0,2g+2})$. In fact we only need to check that $(2g+1)(2g+2)r/d$ is even. In that case we would have $(-1)^{(2g+1)(2g+2)r/d}=1$ (see proof of Proposition \ref{BrRep}). In the end of Section 3, we gave two formulas for $d,r$ in terms of $g$. A direct calculation shows that
\[ (2g+1)(2g+2)\frac{r}{d} = (g+1)(g+2) \]
which is an even number.

\paragraph{Case $\textbf{n}$ even.} It is convenient to change $q^{1/d}$ to $t$ to obtain
\[ J(H^n_{i})= t^{2nr} + e_i((-1)^n t^{2n(r-d)} -t^{2nr}).  \]
We let $t^d$ be $n$ root of unity. Then $t^{(d-2r)n}J(T^n_{c_i})=1$. In this case we denote $t^{(d-2r)n}J(H_i)$ by $J'(H_i)$. We want to show that the map
\[J' : \mathrm{Mod}(\Sigma_{0,2g+2}) / \mathcal{N}(H^n_{i}) \rightarrow \mathrm{GL}_d(\mathbb{C}) \]
is a well defined representation. By Lemma \ref{LongEl} the number $(d-2r)(2g+1)(2g+2)$ is an integer multiple of $d$. Therefore $n(d-2r)(2g+1)(2g+2)$ is a multiple of $nd$. Since $t^d$ is a root of unity, then $J'(H_{i})$ satisfy the relations of $\mathrm{Mod}(\Sigma_{0,2g+2})$ (see proof of Proposition \ref{BrRep}).

We note here that the representations $J'$ we constructed are not necessarily irreducible because the parameter $t$ is a root of unity. But this will not be problem because our proof in the next section does not require irreducible representations for $\mathrm{Mod}(\Sigma_{0,2g+2})/ \langle H^n_{i} \rangle$.

\begin{proof}[Proof of Theorem A]
Let $\Sigma_{0,2g+2}$ be a sphere with $2g+2$ marked points. We recall that $\mathrm{Mod}(\Sigma_{0,2g+2})$ is generated by the half-twists $H_i$ for $i\leq 2g+1$. We will prove that $A=(H_1H_2)^6 H_3 (H_1H_2)^6 H^{-1}_3$ has infinite order in $\mathrm{Mod}( \Sigma_{0,2g+2}) / \mathcal{N}(H^n_i)$. In fact we will prove that $J'(A)$ has infinite order in $\mathrm{GL}_d(\mathbb{C})$, and $\mathrm{GL}_d(\mathbb{C})$ respectively, by showing that if $\mu$ is an eigenvalue of either $J'(A)$, then $\mu^n$ will never be trivial for any $n\geq 4$; consequently this means that $J'(A)^n$ will never be the identity matrix for $n\geq 4$.

\paragraph{Case $\textbf{n}$ even.} By the construction of the representation we have
\[ J'(A) = t^{\frac{12}{2}(2r -d)+(d-2r)n} \pi(\sigma_1 \sigma_2)^6 \pi(\sigma_3) \pi(\sigma_1 \sigma_2)^6 \pi(\sigma_3)^{-1}. \]

We denote by $C$ the matrix $ \pi(\sigma_1 \sigma_2)^6 \pi(\sigma_3) \pi(\sigma_1 \sigma_2)^6 \pi(\sigma_3)^{-1}$. We want to prove that $C$ has infinite order. By Theorem 3.8 we have that the first $5\times 5$ block of $C$ are the same for any $g\geq 2$. Denote this $5\times 5$ block by $C'$. We note that $C'$ has $5$ eigenvalues (2 distinct, and 1 repeated).

%For $n=4$, if we set $t^d=i$ (fourth root of unity), then the diagonal entries of $C''$ are equal to $i$, and the nonzero entries different from the diagonal ones are equal to 1. Hence in that case, $C''$ has infinite order. 
For $n=6$, if $t^d=\mathrm{exp}(\pi i /3)$, then the absolute value of one of the eigenvalues of $C'$ is approximately equals to $9.8989795$. Hence it has infinite order. It follows that if $n$ is even and it is divided by $3$ then the eigenvalue has infinite order. For $n = 2(3k \pm 1)>6$, put $t^d = \exp((4 \pi i k)/n)$. If $k \rightarrow \infty $ then the absolute value of the same eigenvalue approximately converges to $9.8989795$, thus it has infinite order. This completes the proof when $n$ even.

\paragraph{Case $\textbf{n}$ odd.} By construction we have
\[ J'(A)=(-1)^{\frac{-12r}{d}} t^{6(2r-d)} \pi(\sigma_1 \sigma_2)^6 \pi(\sigma_3) \pi(\sigma_1 \sigma_2)^6 \pi(\sigma_3)^{-1}. \]
We denote by $C$ the matrix $\pi(\sigma_1 \sigma_2)^6 \pi(\sigma_3) \pi(\sigma_1 \sigma_2)^6 \pi(\sigma_3)^{-1}$. Again by Theorem 3.8 we have that the first $5\times 5$ block of $C$ are the same for any $g\geq 2$. We denote this $5\times 5$ block by $C'$. The matrix $C'$ has 3 eigenvalues (2 distinct and 1 repeated).

Let $(-1)^nt^d$ be an $n^{th}$ root of unity. We set $t^d=\exp((3 k \pi i)/ n)$ where $n=(4k \pm 1)$ if $k$ is odd, and $t^d=(-1)^{-n}\exp(((3k-1) \pi i)/ n)$ where $n=(4k \pm 1)$ if $k$ is even. In both cases if $k \rightarrow \infty $ then the absolute value of one of the eigenvalues of $C'$ approximately converges rapidly to $9.5521659$, thus it has infinite order. This completes the proof when $n$ odd.  
\end{proof}

Now we can use Theorem A to prove Corollary \ref{coxthem}.

\begin{proof}[Proof of Corollary \ref{coxthem}]
We want to prove that if $\mathcal{N}(\sigma^n_i)$ is the normal closure of $\sigma^n_i$ in $B_{m}$, then $B_m / \mathcal{N}(\sigma^n_i)$ has infinite order when $n \geq 4$, and $m \geq 5$.

Recall the representations $J' : \mathrm{Mod}(\Sigma_{0,2g+2}) / \mathcal{N}(H^{n}_i) \rightarrow \mathrm{GL}_d(\mathbb{C}) $, when $n$ is even, and $J' : \mathrm{Mod}(\Sigma_g) /  \mathcal{N}(H^{n}_{i}) \rightarrow \mathrm{GL}_d(\mathbb{C}) $ when $n$ is odd. By the surjective homomorphism $B_{2g+2} \twoheadrightarrow \mathrm{Mod}(\Sigma_{0,2g+2})$ defined by $\sigma_i \mapsto H_{i}$ we have
\[ B_{2g+2} / \mathcal{N}(\sigma^n_i) \twoheadrightarrow \mathrm{Mod}(\Sigma_{0,2g+2}) /\mathcal{N}(H^n_{i}) \rightarrow \mathrm{GL}_d(\mathbb{C}) \]
when $n$ is even, and
\[ B_{2g+2} / \mathcal{N}(\sigma^n_i) \twoheadrightarrow \mathrm{Mod}(\Sigma_{0,2g+2}) / \mathcal{N}(H^n_{i}) \rightarrow \mathrm{GL}_d(\mathbb{C})  \]
%when $n$ is odd.

\paragraph{Case $\textbf{m}$ even.} In this case the theorem follows by the surjectivity of
\[ B_{2g+2} / \mathcal{N}(\sigma^n_i) \twoheadrightarrow \mathrm{Mod}(\Sigma_{0,2g+2}) / \mathcal{N}(H^n_{i}). \]

\paragraph{Case $\textbf{m}$ odd.} Here we want to prove that $B_{2g+1} / \mathcal{N}(\sigma^n_i)$ has infinite order. If we restrict the above representations to $B_{2g+1}$ we get
\[ B_{2g+1} / \mathcal{N}(\sigma^n_i) \rightarrow \mathrm{GL}_d(\mathbb{C}) \]
when $n$ is even, and
\[ B_{2g+1} / \mathcal{N}(\sigma^n_i) \rightarrow \mathrm{GL}_d(\mathbb{C})  \]
when $n$ is odd. These representations above are well defined, by the restriction formula described in Section 3.1. But by the proof of the Theorem A, the elements $(\sigma_1 \sigma_2)^6 \sigma_3 (\sigma_1 \sigma_2)^6(\sigma_3)^{-1}$, and $(\sigma_1 \sigma_2)^6 \sigma_3 (\sigma_1 \sigma_2)^6$ have infinite order in $B_{2g+1} / \mathcal{N}(\sigma^n_i)$.

To complete the proof, we denote by $\mathcal{N'}(\sigma^n_i)$ the normal closure of $\sigma^n_i$ in $B_{4}$, and we denote by $\mathcal{N}(\sigma^n_i)$ the normal closure of $\sigma^n_i$ in $B_{5}$.

We have that $B_4/(\mathcal{N}(\sigma^n_i) \cap B_4) <B_5/ \mathcal{N}(\sigma^n_i)$, and $(\mathcal{N}(\sigma^n_i) \cap B_4)/ \mathcal{N'}(\sigma^n_i) \lhd B_4/ \mathcal{N'}(\sigma^n_i)$. By the third isomorphism theorem we get a surjective homomorphism
\[ B_4/ \mathcal{N'}(\sigma^n_i) \rightarrow B_4/(\mathcal{N}(\sigma^n_i) \cap B_4) <B_5/ \mathcal{N}(\sigma^n_i). \]
But again the elements $(\sigma_1 \sigma_2)^6 \sigma_3 (\sigma_1 \sigma_2)^6$ and $(\sigma_1 \sigma_2)^6 \sigma_3 (\sigma_1 \sigma_2)^6(\sigma_3)^{-1}$ are in $B_4$. Hence, $B_4/ \mathcal{N'}(\sigma^n_i)$ has infinite order. 
\end{proof}

\section{Proof of Theorem B}

In this section we prove the Theorem B. Let $\mathcal{N}(H_i)$ denote the normal closure of a half-twist $H_i$ in $\mathrm{Mod}(\Sigma_{0,2g+2})$. The representation $J: \mathrm{Mod}(\Sigma_{0,2g+2}) \rightarrow \mathrm{End}(V_{\lambda})$ is defined by $J(H_i)=q^{(2r-d)/d}\pi_{\lambda}(\sigma_i)$, where $\sigma_i \in B_{2g+2}$, $\lambda$ is a Young diagram as in Figure 8, and $d$ is the dimension of the representation.

\paragraph{Proof of Theorem B} Consider Equation (\ref{equ}) from previous section. For $n$ even, let $q^{2/d}$ be a $n^{th}$ root of unity. Then we have that $J(H^n_i)$ is trivial. By denoting $J$ by $J'$ we have a well defined linear representation
\[J':\mathrm{Mod}(\Sigma_{0,2g+2}) / \mathcal{N}(H^n_i) \rightarrow \mathrm{GL}_{d}(\mathbb{C}).\]

For $n$ odd, let $-q^2$ be an $n^{th}$ root of unity. In this case $(-1)^{-r/d}J(H^n_i)$ is trivial (see proof of Theorem A). We set $(-1)^{-r/d}J(H_i)=J'(H_i)$ and we have
\[ J':\mathrm{Mod}(\Sigma_{0,2g+2}) / \mathcal{N}(H^n_i) \rightarrow \mathrm{GL}_{d}(\mathbb{C}). \]

%\paragraph{case $\textbf{n}$ odd.} Let $(-1)^n q^{2/d}$ be a $n^{th}$ root of unity. Then we have that $(-1)^{r}J(H^n_i)$ is trivial. By denoting $(-1)^{r}J(H_i)$ by $J'(H_i)$ we have a well defined linear representation
%\[J':\mathrm{Mod}(\Sigma_{0,2g+2}) / \mathcal{N}(H^n_i) \rightarrow \mathrm{GL}_{d}(\mathbb{C}[t^{\pm 1/d}]).\]

Our aim is to prove that the group generated by the elements $J'(H^2_1),J'(H^2_2)$ contains a free nonabelian subgroup. By Young's lattice described in Section 3.1, if we restrict the Jones representation to the subgroup generated by $H_1,H_2$ then the representation $J$ reduces into a direct sum of subrepresentations containing the representation labeled by $\lambda=${\tiny\yng(2,1)}. Therefore, the elements $q^{(2r-d)/d}\pi_{\lambda}(\sigma_i)=\pi'_{\lambda}(\sigma_i)$ where $i=1,2$ are contained in the image of $J$. By Section 2.2 the map $\pi_{\lambda}$ is the Burau representation. If $q^{2}$ is not a primitive root of unity of order in the set $\{ 1, 2, 3, 4, 6, 10 \}$, then $\pi_{\lambda}(PB_3)$ contains a free nonabelian subgroups \cite[Lemma 3.9]{FK}.

\paragraph{Case $\textbf{n}$ even.} If $q^{2/d}$ be $n^{th}$ nonprimitive root of unity such that $n\notin \{ 2,4,6,10 \}$, then $\pi'_{\lambda}(PB_3)<J'(G)$ contains a free nonabelian subgroups.

\paragraph{Case $\textbf{n}$ odd.} If we consider $-q^2$ be $n^{th}$ nonprimitive root of unity such that $n\notin \{ 1,3,5 \}$, then $\pi'_{\lambda}(PB_3)<J'(G)$ contains a free nonabelian subgroups.

\nocite{*}
\bibliographystyle{plain}
\bibliography{Jones_representation.bib}

\end{document}